\documentclass[12pt]{amsart}
\usepackage{amsmath}
\usepackage{amsxtra}
\usepackage{enumitem}
\usepackage{amscd}
\usepackage{scalerel}
\usepackage{tikz-cd}
\usepackage{amsthm}
\usepackage{hyperref}
\usepackage[T1]{fontenc}
\usepackage{amsfonts}
\usepackage{amssymb}
\usepackage{eucal}
\usepackage[all]{xypic}
\usepackage{color} 
\newtheorem{Theo}[subsubsection]{Theorem}
\newtheorem{Theor}{Theorem}
\newtheorem{Theore}{Theorem}

\newtheorem{cor}[subsubsection]{Corollary}

\theoremstyle{definition}

\newtheorem{rem}[subsubsection]{Remark}

\newtheorem{question}[subsection]{Question}
\newtheorem{Prop}[subsubsection]{Proposition}
\newtheorem{Obs}[subsubsection]{Observation}

\newtheorem{Lemm}[subsubsection]{Lemma}

\newtheorem{ex}[subsection]{Exercise}

\usepackage{yhmath}
\DeclareSymbolFont{largesymbols}{OMX}{yhex}{m}{n}
\DeclareMathAccent{\widetilde}{\mathord}{largesymbols}{"65}
\def\msquare{\mathord{\scalerel*{\Box}{gX}}}

\newcommand{\bp}{\begin{Prop}}
\newcommand{\ep}{\end{Prop}}

\newcommand{\bl}{\begin{Lemm}}
\newcommand{\el}{\end{Lemm}}
\newcommand{\bex}{\begin{ex} \rm}
\newcommand{\eex}{\end{ex}}
\newcommand{\bt}{\begin{Theo}}
\newcommand{\et}{\end{Theo}}
\newcommand{\bq}{\begin{question}}
\newcommand{\eq}{\end{question}}

\newcommand{\bc}{\begin{cor}}
\newcommand{\ec}{\end{cor}}
\newcommand{\bob}{\begin{Obs}}
\newcommand{\eob}{\end{Obs}}

\newcommand{\nc}{\newcommand}
\nc{\renc}{\renewcommand}
\nc{\ssec}{\subsection}
\nc{\sssec}{\subsubsection} 
\nc\ol{\overline}
\nc\wt{\widetilde}
\nc\wh{\widehat}
\nc\tboxtimes{\wt{\boxtimes}}

\renc{\d}{{\delta}}
\nc{\Aa}{{\mathbb{A}}}
\nc{\Bb}{{\mathbb{B}}}
 \nc{\Gg}{{\mathbb{G}}}  
\nc{\Hh}{{\mathbb{H}}}
 \nc{\Nn}{{\mathbb{N}}}
\nc{\Pp}{{\mathbb{P}}}
\nc{\Rr}{{\mathbb{R}}}
\newcommand{\F}{\mathbb{F}}
\nc{\BV}{{\mathbb{V}}}
\nc{\BW}{{\mathbb{W}}}
\newcommand{\Z}{\mathbb{Z}}
\newcommand{\N}{\mathbb{N}}

\nc{\Qq}{{\mathbb{Q}}}
\nc{\Ss}{{\mathbb{S}}}
\nc{\Cc}{{\mathbb{C}}}
\nc{\Ff}{{\mathbb{F}}}
 \nc{\EL}{{L_\infty}}

\nc{\CA}{{\mathcal{A}}}
\nc{\CB}{{\mathcal{B}}}

\nc{\CE}{{\mathcal{E}}}
\nc{\CF}{{\mathcal{F}}}

\nc{\Las}{\mathsf{Las}}
\nc{\CG}{{\mathcal{G}}}

\nc{\CL}{{\mathcal{L}}}
\nc{\CC}{{\mathcal{C}}}
\nc{\CM}{{\mathcal{M}}}

\nc{\CN}{{\mathcal{N}}}
\nc{\Oog}{{\mathbb{O}}}
\nc{\Oo}{{\mathcal{O}}}
\nc{\CP}{{\mathcal{P}}}
\nc{\CQ}{{\mathcal{Q}}}
\nc{\CR}{{\mathcal{R}}}
\nc{\CS}{{\mathcal{S}}}
\nc{\CT}{{\mathcal{T}}}
\nc{\CU}{{\mathcal{P}}}

\nc{\CV}{{\mathcal{V}}}
 
\nc{\CW}{{\mathcal{W}}}
\nc{\CZ}{{\mathcal{Z}}}

\nc{\cM}{{\check{\mathcal M}}{}}
\nc{\csM}{{\check{\mathcal A}}{}}
\nc{\oM}{{\overset{\circ}{\mathcal M}}{}}
\nc{\obM}{{\overset{\circ}{\mathbf M}}{}}
\nc{\oCA}{{\overset{\circ}{\mathcal A}}{}}
\nc{\obA}{{\overset{\circ}{\mathbf A}}{}}
\nc{\ooM}{{\overset{\circ}{M}}{}}
\nc{\osM}{{\overset{\circ}{\mathsf M}}{}}
\nc{\vM}{{\overset{\bullet}{\mathcal M}}{}}
\nc{\nM}{{\underset{\bullet}{\mathcal M}}{}}
\nc{\oD}{{\overset{\circ}{\mathcal D}}{}}
\nc{\obD}{{\overset{\circ}{\mathbf D}}{}}
\nc{\oA}{{\overset{\circ}{\mathbb A}}{}}
\nc{\op}{{\overset{\bullet}{\mathbf p}}{}}
\nc{\cp}{{\overset{\circ}{\mathbf p}}{}}
\nc{\oU}{{\overset{\bullet}{\mathcal U}}{}}
\nc{\oZ}{{\overset{\circ}{\mathcal Z}}{}}
\nc{\ofZ}{{\overset{\circ}{\mathfrak Z}}{}}
\nc{\oF}{{\overset{\circ}{\fF}}}

\nc{\fa}{{\mathfrak{a}}}
\nc{\fb}{{\mathfrak{b}}}
\nc{\fg}{{\mathfrak{g}}}
\nc{\fgt}{{\fg}_!}
\nc{\fgl}{{\mathfrak{gl}}}
\nc{\fh}{{\mathfrak{h}}}
\nc{\fj}{{\mathfrak{j}}}
\nc{\fm}{{\mathfrak{m}}}
\nc{\ft}{{\mathfrak{t}}}
\nc{\fn}{{\mathfrak{n}}}
\nc{\fu}{{\mathfrak{u}}}
\nc{\fp}{{\mathfrak{p}}}
\nc{\fr}{{\mathfrak{r}}}
\nc{\fs}{{\mathfrak{s}}}
\nc{\fsl}{{\mathfrak{sl}}}
\nc{\hsl}{{\widehat{\mathfrak{sl}}}}
\nc{\hgl}{{\widehat{\mathfrak{gl}}}}
\nc{\hg}{{\widehat{\mathfrak{g}}}}
\nc{\chg}{{\widehat{\mathfrak{g}}}{}^\vee}
\nc{\hn}{{\widehat{\mathfrak{n}}}}
\nc{\chn}{{\widehat{\mathfrak{n}}}{}^\vee}

\nc{\fA}{{\mathfrak{A}}}
\nc{\fB}{{\mathfrak{B}}}
\nc{\fD}{{\mathfrak{D}}}
\nc{\fE}{{\mathfrak{E}}}
\nc{\fF}{{\mathfrak{F}}}
\nc{\fG}{{\mathfrak{G}}}
\nc{\fK}{{\mathfrak{K}}}
\nc{\fL}{{\mathfrak{L}}}
\nc{\fM}{{\mathfrak{M}}}
\nc{\fN}{{\mathfrak{N}}}
\nc{\fP}{{\mathfrak{P}}}
\nc{\fU}{{\mathfrak{U}}}
\nc{\fV}{{\mathfrak{V}}}
\nc{\fZ}{{\mathfrak{Z}}}
\newcommand{\Q}{\mathbb{Q}}
\nc{\bb}{{\mathbf{b}}}
\nc{\bd}{\partial}
\nc{\be}{{\mathbf{e}}}
\nc{\bj}{{\mathbf{j}}}
\nc{\bn}{{\mathbf{n}}}
\nc{\bF}{{\mathbf{F}}}
\nc{\bu}{{\mathbf{u}}}
\nc{\bv}{{\mathbf{v}}}
\nc{\bx}{{\mathbf{x}}}
\nc{\bs}{{\mathbf{s}}}
\nc{\by}{{\bar{y}}}
\nc{\bw}{{\mathbf{w}}}
\nc{\bA}{{\mathbf{A}}}
\nc{\bK}{{\mathbf{K}}}
\nc{\bI}{{\mathbf{I}}}
\nc{\bB}{{\mathbf{B}}}
\nc{\bG}{{\mathbf{G}}}

\nc{\bD}{{\mathbf{D}}}
\nc{\bP}{{\mathbf{P}}}
\nc{\bH}{{\mathbf{H}}}
\nc{\bM}{{\mathbf{M}}}
\nc{\bN}{{\mathbf{N}}}
\nc{\bV}{{\mathbf{V}}}
\nc{\bU}{{\mathbf{U}}}
\nc{\bL}{{\mathbf{L}}}

\nc{\bW}{{\mathbf{W}}}
\nc{\bX}{{\mathbf{X}}}
\nc{\bY}{{\mathbf{Y}}}
\nc{\bZ}{{\mathbf{Z}}}
\nc{\bS}{{\mathbf{S}}}
\nc{\bSi}{{\bar{\Sigma}}}
\nc{\sA}{{\mathsf{A}}}
\nc{\sB}{{\mathsf{B}}}
\nc{\sC}{{\mathsf{C}}}
\nc{\sD}{{\mathsf{D}}}
\nc{\sF}{{\mathsf{F}}}
\nc{\sG}{{\mathsf{G}}}
\nc{\sK}{{\mathsf{K}}}
\nc{\sM}{{\mathsf{M}}}
\nc{\sO}{{\mathsf{O}}}
\nc{\sQ}{{\mathsf{Q}}}
\nc{\sP}{{\mathsf{P}}}
\nc{\sZ}{{\mathsf{Z}}}
\nc{\sfp}{{\mathsf{p}}}
\nc{\sr}{{\mathsf{r}}}
\nc{\sg}{{\mathsf{g}}}
\nc{\sff}{{\mathsf{f}}}
\nc{\sfb}{{\mathsf{b}}}
\nc{\sfc}{{\mathsf{c}}}
\nc{\sd}{{\ltimes}} 

\nc{\tH}{{\widetilde{H}}}
\nc{\tA}{{\widetilde{\mathbf{A}}}}
\nc{\tB}{{\widetilde{\mathcal{B}}}}
\nc{\tg}{{\widetilde{\mathfrak{g}}}}
\nc{\tG}{{\widetilde{G}}}

\nc{\TM}{{\widetilde{\mathbb{M}}}{}}
\nc{\tO}{{\widetilde{\mathsf{O}}}{}}
\nc{\tU}{\widetilde{U}}
\nc{\TZ}{{\tilde{Z}}}
\nc{\tx}{{\tilde{x}}}
\nc{\tq}{{\tilde{q}}}

\nc{\tfP}{{\widetilde{\mathfrak{P}}}{}}
\nc{\tz}{{\tilde{\zeta}}}
\nc{\tmu}{{\tilde{\mu}}}

  \nc{\vol}{{\mathop{\operatorname{\rm vol\,}}}}
  
  \nc{\gal}{{\mathop{\operatorname{\rm Gal\,}}}}
  \nc{\cl}{{\mathop{\operatorname{\rm cl}}}}
  \nc{\disc}{{\mathop{\operatorname{\rm disc}}}}
  \nc{\Sym}{{\mathop{\operatorname{\rm Sym}}}}
   \nc{\Aut}{{\mathop{\operatorname{\rm Aut}}}}
 \nc{\Spec}{{\mathop{\operatorname{\rm Spec}}}}
  \nc{\spec}{{\mathop{\operatorname{\rm Spec}}}}
\nc{\Ker}{{\mathop{\operatorname{\rm Ker}}}}
 \nc{\dom}{{\mathop{\operatorname{\rm dom}}}}
\nc{\End}{{\mathop{\operatorname{\rm End}}}}
 \nc{\Hom}{\operatorname{\Hom}}
 \nc{\GL}{{\mathop{\operatorname{\rm GL}}}}
 \nc{\Id}{{\mathop{\operatorname{\rm Id}}}}
 \nc{\rk}{{\mathop{\operatorname{\rm rk}}}}
 \nc{\length}{{\mathop{\operatorname{\rm length}}}}
\nc{\supp}{{\mathop{\operatorname{\rm supp} \, }}}
\nc{\val}{{\rm val}}
\nc{\res}{{\mathop{\operatorname{\rm res}}}}

\def\Ind#1#2#3{{#1} {\downarrow}_{#3} {#2} }

\nc{\seq}[1]{\stackrel{#1}{\sim}}

\def\beq#1{\begin{equation} \label{ #1}}
\def\eeq{\end{equation}}

\def\prf{\begin{proof}}
\def\pv{\end{proof} }
 \def\eprf{\end{proof} }

 \renc{\b}{{\beta}}

\def\Ind#1#2{#1\setbox0=\hbox{$#1x$}\kern\wd0\hbox to 0pt{\hss$#1\mid$\hss}
\lower.9\ht0\hbox to 0pt{\hss$#1\smile$\hss}\kern\wd0}

\usepackage{color}
\usepackage{hyperref}

\usepackage{url}

\title{Diophantine problems over $\Z^{ab}$ modulo prime numbers}
\author{Konstantinos Kartas }
\newcommand{\Addresses}{{
  \bigskip
  \footnotesize

\textsc{Mathematical Institute, Woodstock Road, Oxford OX2 6GG.}\par\nopagebreak
  \textit{E-mail address}: \texttt{kartas@maths.ox.ac.uk}
}}

\begin{document}


\maketitle
\begin{abstract}
Let $\Z^{ab}$ be the ring of integers of $\Q^{ab}$, the maximal abelian extension of $\Q$. We show that there exists an algorithm to decide whether a system of equations and inequations, with integer coefficients, has a solution in $\Z^{ab}$ modulo every rational prime. 
\end{abstract}
\setcounter{tocdepth}{1}
\tableofcontents

\section*{Introduction} 
In \cite{AxFin}, Ax has exhibited an algorithm for deciding whether a given system of polynomial equations (and inequations) over $\Z$ has a solution modulo $p$, for all prime numbers $p\in \mathbb{P}$. Phrased differently, the existential theory of all prime fields $\{\Z/p\Z:p\in \mathbb{P}\}$ in the language of rings $L_r$ is decidable. One year later, Ax \cite{Ax} generalized this by showing that the full first-order theory of $\{\Z/p\Z:p\in \mathbb{P}\}$ is decidable. This was perhaps in sharp contrast with the fact that no algorithm exists for determining whether a system of polynomial equations has a solution in $\Z$ and that the corresponding problem for $\Q$, i.e. Hilbert's tenth problem over $\Q$, is widely open. \\  

In this paper, we study an analogous problem for $\Z^{ab}$, where $\Z^{ab}$ is the ring of integers of $\Q^{ab}$, the maximal abelian extension of $\Q$. It is not known whether the existential theory of $\Z^{ab}$ is decidable in $L_r$. For some comments on the \textit{conjectural} undecidability of $\Z^{ab}$, see the last bullet point on Section 6.3, pg.191 in \cite{JK}. \\
Our main goal is to prove the following theorem:
\begin{Theor} \label{main}
The existential theory of $\{\Z^{ab}/p\Z^{ab}:p\in \mathbb{P}\}$ is decidable in $L_r$.
\end{Theor}

An important difference with the existential theory of $\{\Z/p\Z:p\in \mathbb{P}\}$, is that $\Z^{ab}/p\Z^{ab}$ is not a field; in fact it is not even a local ring but merely a direct limit of semi-local rings. In particular, one cannot simply eliminate inequations $\exists x (g(x)\neq 0)$ by replacing them with $\exists x,y(yg(x)-1=0)$. Similarly, one cannot replace $\exists x(g_1(x)\neq 0\land ... \land g_n(x)\neq 0)$ with $\exists x (g_1(x)\cdot ...\cdot g_n(x)\neq 0)$, since $\Z^{ab}/p\Z^{ab}$ is not an integral domain. Moreover, the ring $\Z^{ab}/p\Z^{ab}$ has lots of nilpotent elements, which arise due to high ramification and make its model theory difficult to understand. Nonetheless, the ring $\Z^{ab}/p\Z^{ab}$ resembles "enough" a valuation ring, so that we can eventually encode our problem inside $ACVF$\footnote{More precisely, inside $\bigcap_{p\in \mathbb{P}}ACVF_{(p,p)}$ (see notation), which is also well-understood. This is explained in the proof of Theorem \ref{mainagain}.}, which is well-understood by work of A. Robinson (see Sections 3.4, 3.5 \cite{vdd} for an exposition of A. Robinson's result).
\section*{Notation}
If $(K,v)$ is a valued field, we denote the valuation ring by
 $\mathcal{O}_v$, the value group by $vK$ and the residue field by $Kv$.\\
We also denote by:\\
$\Q^{ab}$: The maximal abelian extension of $\Q$, i.e. the maximal algebraic extension $K/\Q$ with $Gal(K/\Q)$ abelian.\\
$\Z^{ab}$: The ring of integers in $\Q^{ab}$, i.e. the integral closure of $\Z$ in $\Q^{ab}$.\\
$k((t))^{1/p^{\infty}}$: This is the perfect hull of $ k((t))$, not to be confused with the Hahn field with residue field $k$ and value group $\frac{1}{p^{\infty}}\Z$.\\
 $k[[t]]^{1/p^{\infty}}$: The valuation ring of $k((t))^{1/p^{\infty}}$ with respect to the $t$-adic valuation.\\
$L_r$: The language of rings, i.e. $\{+,-,\cdot,0,1\}$.\\
$L_{oag}$: The language of ordered abelian groups, i.e. $\{+,<,0\}$.\\
$L_{val}$: The language of valued fields, construed as a three-sorted language with sorts for the valued field, the value group, the residue field and symbols for the valuation, the residue map and the place of the valuation. \\
ACVF: The theory of algebraically closed valued fields in $L_{val}$ (see Section 3.4 \cite{vdd}).\\
ACVF$_{(q,p)}$: The theory of algebraically closed valued fields of characteristic $(char(K),char(Kv))=(q,p)$ in $L_{val}$ (see Section 3.5 \cite{vdd}).\\
When $T$ is a theory in the language $L$, we denote by $T_{\exists}$ the set of existential sentences $\phi \in L $ such that $T\models \phi$.
\section{Modulo $p$ computations}
In this section we fix a prime $p\in \mathbb{P}$ and start by analyzing the ring $\Z^{ab}/p\Z^{ab}$. By Kronecker-Weber, the ring $\Z^{ab}$ can be written as $\Z[\zeta_{\infty}]=\Z[\zeta_{\infty '},\zeta_{p^{\infty}}]$, where $\Z[\zeta_{\infty '}]=\Z[\{\zeta_n:(n,p)=1\}]$ and $\Z[\zeta_{p^{\infty}}]=\Z[\{\zeta_{p^n}:n\in \N\}]$. The plan will be to first analyze $R_p=\Z[\zeta_{\infty '}]/p\Z[\zeta_{\infty '}]$ and then understand $\Z^{ab}/p\Z^{ab}$ as an $R_p$-algebra. Eventually, we shall be able to convert existential sentences about $\Z^{ab}/p\Z^{ab}$ to existential sentences about a certain $\bar \F_p$-algebra (see Corollary \ref{corconv}).
\subsection{Preliminaries}
We collect here some preliminary facts which will make the exposition in the following sections easier.
\bl \label{ringiso}
Let $R$ be any ring, $m\in \N$ and $\kappa$ be a cardinal number. Then we have the following ring isomorphism $R^{\kappa}[x]/(x^m)\cong (R[x]/(x^m))^{\kappa}$.
\el 
\begin{proof}
Consider $f:R^{\kappa}[x]/(x^m) \to (R[x]/(x^m))^{\kappa}$ which sends 
$$ a_0+a_1x+...+a_{m-1}x^{m-1} +(x^m) \mapsto (a_0(i)+a_1(i)x+...+a_{m-1}(i) x^{m-1}+(x^m))_{i\in \kappa}$$
and $g:(R[x]/(x^m))^{\kappa}\to R^{\kappa}[x]/(x^m)$ mapping 
$$(a_0^{(i)}+a_1^{(i)}x+...+a_{m-1}^{(i)}x^{m-1}+(x^m) )_{i\in \kappa} \mapsto  a_0+a_1x+...+a_{m-1}x^{m-1}+(x^m)$$
where $a_j:\kappa\to R$ is defined by $a_j(i)=a_j^{(i)}$. It is clear that $f$ and $g$ are ring homomorphisms and inverses of each other. 
\end{proof}
\bl \label{DNF}
Every quantifier-free formula in $L_r$ is logically equivalent to one of the form $\bigvee_{i=1}^N \bigwedge_{j=1}^n  (f_{ij}(x)  \msquare 0)$, where $\msquare$ stands for $"="$ or $"\neq "$ and $f_{ij}(x)\in \Z[x]$ with $x=(x_1,...,x_m)$ and $m\in \N$.
\el 
\begin{proof}
This is a special case of the fact that every quantifier-free formula in a language $L$ is logically equivalent to one in disjunctive normal form (see pg.42 \cite{Hod}).
\end{proof}

While injective limits do not commute with infinite products (in general), we have the following:
\bl \label{directlim}
Let $\langle R_i, \phi_{ij} \rangle_{i,j\in I}$ be a direct system of rings over some directed set $\langle I,\leq \rangle$ and $R=\varinjlim R_i $. Suppose that $\phi_i:R_i\to R$ is an injective ring homomorphism for each $i\in I$. Let also $\kappa\geq \aleph_0$ be a cardinal number. Then we may form a new direct system $\langle (R_i)^{\kappa}, \phi_{ij}^{\kappa} \rangle_{i,j\in I}$ and we have that $R^{\kappa}\equiv_{\exists} \varinjlim (R_i)^{\kappa}$ in $L_r$.
\el 
\begin{proof} 
By Lemma \ref{DNF}, it suffices to check that $R^{\kappa}\models \phi \iff \varinjlim (R_i)^{\kappa}\models \phi$ for existential sentences $\phi \in L_r$ of the form 
$$\phi= \exists x( \bigwedge_{1\leq i\leq n} f_i(x)=0\land g_i(x)\neq 0)$$
where $x=(x_1,...,x_m)$, $f_i(x),g_i(x)\in \Z[x]$ for $i=1,...,n$. \\ 
Suppose that $R^{\kappa} \models \phi$. We will then have that $R^n\models \phi$, where the number $n$ is the same as in the definition of $\phi$. Since $R^{n}\cong \varinjlim(R_i)^{n}$, we will also have that $\varinjlim(R_i)^{n} \models \phi  $ and thus $\varinjlim (R_i)^{\kappa} \models \phi$. The other direction is clear since we have a natural ring embedding $\varinjlim (R_i)^{\kappa}\hookrightarrow R^{\kappa}$.
\end{proof}
Let $R$ be a ring. Two ideals $\mathfrak{a},\mathfrak{b}$ of $R$ are said to be relatively prime if $\mathfrak{a}+\mathfrak{b}=R$. Recall the following classical result:
\bt [Chinese Remainder Theorem] \label{CRT}
Let $\mathfrak{a}_1,...,\mathfrak{a}_n$ be ideals in a ring $R$, relatively prime in pairs. We then have a short exact sequence 
$$0\to \mathfrak{a}\to R \to \prod_{i=1}^n R/\mathfrak{a}_i \to 0$$
where $\mathfrak{a}=\bigcap_{i=1}^n \mathfrak{a}_i$.
\et 
\begin{proof}
See Theorem 17, pg. 265 \cite{DF}.
\end{proof}

\subsection{The unramified part} \label{unrampart}
Fix a prime number $p\in \mathbb{P}$. Let $\Z[\zeta_{\infty '}]=\Z[\{\zeta_n:(n,p)=1\}]$ and 
$$R_p=\Z[\zeta_{\infty'}]/p\Z[\zeta_{\infty'}]\cong \varinjlim\Z[\zeta_{p^n-1}]/p \Z[\zeta_{p^n-1}]$$
The limit is meant to be the direct limit of the \textit{injective}\footnote{One has to check that $p\Z[\zeta_{p^n-1}]=\Z[\zeta_{p^n-1}]\cap p\Z[\zeta_{p^m-1}]$, when $n|m$. Write $p\Z[\zeta_{p^n-1}]=\mathfrak{p}_1\cap...\cap \mathfrak{p}_r$ and $p\Z[\zeta_{p^m-1}]= \mathfrak{q}_1\cap...\cap \mathfrak{q}_l$, where the $\mathfrak{p}_i$'s (resp. $\mathfrak{q}_i$'s) are the prime ideals of $\Z[\zeta_{p^n-1}]$ (resp. $\Z[\zeta_{p^m-1}]$) lying above $(p)\subset \Z$. Moreover, each $\mathfrak{p}_i$ lies under some $\mathfrak{q}_j$, using that $ \Z[\zeta_{p^n-1}]\subset \Z[\zeta_{p^m-1}]$ is integral (Theorem 26 (2), pg.694 \cite{DF}). It follows that $p\Z[\zeta_{p^n-1}]=\Z[\zeta_{p^n-1}]\cap p\Z[\zeta_{p^m-1}]$. } system of rings
 $\langle  \Z[\zeta_{p^n-1}]/p\Z[\zeta_{p^n-1}],\phi_{nm} \rangle$ over the directed poset $\langle \N,| \rangle$ with
$$\phi_{nm}:\Z[\zeta_{p^n-1}]/p \Z[\zeta_{p^n-1}]\to \Z[\zeta_{p^m-1}]/p \Z[\zeta_{p^m-1}]$$
being induced by the natural inclusion $ \Z[\zeta_{p^n-1}]\hookrightarrow \Z[\zeta_{p^m-1}]$, when $n|m$.
\bl \label{modulopfinite}
Let $m\in \N$ and $q=p^m$. We have an isomorphism of rings
$$\Z[\zeta_{q-1}]/p\Z[\zeta_{q-1}]\cong \F_q\times...\times \F_q=\F_q^r$$
Moreover, we have that $r=\frac{\phi(q-1)}{m}$.
\el 
\begin{proof}
Since $p$ is unramified in $\Z[\zeta_{q-1}]$, its (unique) factorization into prime ideals has the following shape
$$(p)=\mathfrak{p}_1\cdot...\cdot \mathfrak{p}_r=\mathfrak{p}_1\cap \mathfrak{p}_2\cap...\cap \mathfrak{p}_r$$
where the $\mathfrak{p}_i$'s are distinct prime ideals in $\Z[\zeta_{q-1}]$, each with residue field $\F_q$. The conclusion thus follows from the Chinese Remainder Theorem (see Theorem \ref{CRT}). The fact that $r=\frac{\phi(q-1)}{m}$ follows from the fundamental equality (see Corollary on pg.20 \cite{Ser}) by the following computation
$$\phi(q-1)=[\Q(\zeta_{q-1}):\Q]=e_p(\Q(\zeta_{q-1})/\Q)\cdot f_p(\Q(\zeta_{q-1})/\Q) \cdot  r=[\F_q:\F_p]\cdot r=m\cdot r$$
\end{proof}

Let $\bar \F_p^{\omega}$ be the ring produced by taking the Cartesian product of $\omega$ copies of $\bar \F_p$.
\bp \label{rpexistential}
We have that $R_p \equiv_{\exists} \bar \F_p^{\omega}$ in the language of rings $L_r$.
\ep  
\begin{proof}
By Lemma \ref{DNF}, it suffices to show that $R_p\models \phi \iff  \bar \F_p^{\omega}\models \phi$, for sentences $\phi \in L_r$ of the form
$$\exists x( \bigwedge_{1\leq i\leq n} f_i(x)=0\land g_i(x)\neq 0)$$
$"\Rightarrow":$ Suppose that $R_p\models \phi$. Since 
$$R_p\cong \varinjlim \Z[\zeta_{p^n-1}]/p\Z[\zeta_{p^n-1}]$$
there exists $m\in \N$ such that $\Z[\zeta_{q-1}]/p\Z[\zeta_{q-1}]\models \phi$, where $q=p^m$.  By Lemma \ref{modulopfinite}, we see that 
$$\F_q\times...\times \F_q\models \phi \Rightarrow \bar \F_p^{\omega}\models \phi $$
since there is a ring embedding $\F_q\times...\times \F_q\hookrightarrow \bar \F_p^{\omega}$.\\
$"\Leftarrow":$ Suppose that $\bar \F_p^{\omega}\models \phi$. Then $\bar \F_p^n\models \phi$, where the $n$ is the same as in the definition of $\phi$. Consequently, there exists $m\in \N$  such that $\F_q^n\models \phi$, where $q=p^m$. Replacing $m$ with a large multiple, if necessary, we may also assume that $r=\frac{\phi(p^m-1)}{m}>n$.\footnote{Indeed, for any $\alpha\in \N^{>1}$ we have that
$$\lim_{N\rightarrow \infty} \frac{\phi(\alpha^{N}-1)}{N}=\infty$$
by using for example that $\phi(n)\geq \frac{\sqrt{n}}{2}$ for all $n\in \N$.}
It then follows from Lemma \ref{modulopfinite}, that 
$$\F_q^n \hookrightarrow \F_q^r\cong \Z[\zeta_{q-1}]/p\Z[\zeta_{q-1}]$$
and since there is a ring embedding $ \Z[\zeta_{q-1}]/p\Z[\zeta_{q-1}] \hookrightarrow R_p$, we get that $R_p\models \phi$.
\end{proof}
\begin{rem}
Note that $ R_p \not \cong \bar \F_p^{\omega}$ because $|R_p|=\aleph_0$ while $| \bar \F_p^{\omega}|=2^{\aleph_0}$.
\end{rem}

\subsection{Abelian algebraic integers modulo $p$}
Our goal in this section is to prove Proposition \ref{resringab}, which is an important step in order to understand $\Z^{ab}/p\Z^{ab}$ for a fixed prime $p\in \mathbb{P}$.\\
We shall first view $\Z^{ab}/p\Z^{ab}$ as an $R_p$-algebra, where $R_p$ is as in Section \ref{unrampart}.
\bl \label{resringrp}
Fix a prime number $p\in \mathbb{P}$ and let $R_p$ be as above. Then we have a ring isomorphism $\Z^{ab}/p\Z^{ab} \cong R_p[t^{1/p^{\infty}}]/t^{p-1}$, where the latter is by definition $\varinjlim R_p[t^{1/p^n}]/t^{p-1}$, with morphisms $\phi_{nm}:R_p[t^{1/p^n}]/t^{p-1}\to R_p[t^{1/p^m}]/t^{p-1}$ induced by the natural inclusion maps $R_p[t^{1/p^n}]\hookrightarrow R_p[t^{1/p^m}]$, for $n\leq m$.
\el 
\begin{proof}
By Kronecker-Weber (Theorem 14.1 \cite{Wash}) and the fact that the ring of integers of $\Q(\zeta_n)$ is $\Z[\zeta_n]$ (Theorem 2.6 \cite{Wash}), we have that $\Z^{ab}=\Z[\zeta_{\infty}]$, which is also equal to $\Z[\zeta_{\infty '},\zeta_{p^{\infty}}]$. Note that the irreducible polynomial of $\zeta_p$ over $\Q(\zeta_{\infty '})$ is the cyclotomic polynomial
$$\Phi_p(x)=x^{p-1}+...+1$$
To see this, note that when $p\nmid n$, we have $[\Q(\zeta_n,\zeta_p):\Q(\zeta_n)] \geq [\Q(\zeta_{np}):\Q(\zeta_n)]=\frac{\phi(np)}{\phi(n)}=\phi(p)=p-1$.\footnote{Recall that $[\Q(\zeta_n):\Q]=\phi(n)$, where $\phi$ is Euler's totient function (Theorem 2.5 \cite{Wash}). }. On the other hand, $[\Q(\zeta_n,\zeta_p):\Q(\zeta_n)] \leq [\Q(\zeta_p):\Q]=p-1$ and therefore $[\Q(\zeta_n,\zeta_p):\Q(\zeta_n)]=p-1$. It follows that the irreducible polynomial of $\zeta_p$ over $\Q(\zeta_{\infty'})$ must be of degree $p-1$ and is therefore $\Phi_p(x)$. Note also that its reduction modulo $p$ is equal to $\bar \Phi_p(x)=(x-1)^{p-1}$.\\
Moreover, given $n>1$, the irreducible polynomial of $\zeta_{p^n}$ over $\Z[\zeta_{\infty '},\zeta_{p^{n-1}}]$ is $x^p-\zeta_{p^{n-1}}$. To see this, note that 
$$p\geq deg(irr_{(\zeta_{p^n},\Q(\zeta_{\infty '}, \zeta_{p^{n-1}}))})\geq e(\Q(\zeta_{\infty '}, \zeta_{p^{n}}))/\Q(\zeta_{\infty '},\zeta_{p^{n-1}}))= $$
$$ [\frac{1}{p^n(p-1)}\Z:\frac{1}{p^{n-1}(p-1)}\Z]= p\Rightarrow irr_{(\zeta_{p^n},\Q(\zeta_{\infty '}, \zeta_{p^{n-1}}))}(x)=x^p-\zeta_{p^{n-1}}$$
We may now compute
$$\Z[\zeta_{\infty '}][\zeta_{p^{\infty}}]/(p)=\Z[\zeta_{\infty'}][x^{1/p^{\infty}}]/(p,\Phi_p(x))=R_p[x^{1/p^{\infty}}]/(\bar \Phi_p(x))=$$ 
$$R_p[x^{1/p^{\infty}}]/(x-1)^{p-1}\stackrel{x-1=t}{\cong} R_p[(t+1)^{1/p^{\infty}}]/t^{p-1}=R_p[t^{1/p^{\infty}}]/t^{p-1}$$
which is what we wanted to show.
\end{proof}

\bl \label{equivomega}
We have $\bar \F_p^{\omega}[t^{1/p^{\infty}}]/t^{p-1}\equiv_{\exists} (\bar \F_p[t^{1/p^{\infty}}]/t^{p-1})^{\omega}$ in the language of rings $L_r$.
\el 
\begin{proof}
Indeed, suppose $\phi$ is an existential sentence in $L_r$. Then we get the following series of equivalences
$$ \bar \F_p^{\omega}[t^{1/p^{\infty}}]/t^{p-1} \models \phi \iff (\exists e\in \N)( \bar \F_p^{\omega}[t^{1/p^{e}}]/t^{p-1} \models  \phi )\iff$$ 
$$ (\exists e\in \N)( \bar \F_p^{\omega}[t]/t^{p^e(p-1)} \models  \phi)$$
Applying Lemma \ref{ringiso} with $R=\bar \F_p$, $m=p^e(p-1)$ and $\kappa=\omega$, this is also equivalent to
$$(\exists e\in \N) ((\bar \F_p[t]/t^{p^e(p-1)})^{\omega} \models  \phi) \iff (\exists e\in \N) ((\bar \F_p[t^{1/p^{e}}]/t^{p-1})^{\omega}) \models \phi$$ 
By Lemma \ref{directlim}, we finally see that this is equivalent to $(\bar \F_p[t^{1/p^{\infty}}]/t^{p-1})^{\omega} \models \phi$, which finishes the proof.
\end{proof}

We may now prove:
\bp \label{resringab}
We have $\Z^{ab}/p\Z^{ab} \equiv_{\exists} (\bar \F_p[t^{1/p^{\infty}}]/t^{p-1})^{\omega}$ in the language of rings $L_r$.
\ep   
\begin{proof}
By Proposition \ref{resringrp}, we will have that $\Z^{ab}/p\Z^{ab} \cong R_p[t^{1/p^{\infty}}]/t^{p-1}$. We first show that $R_p[t^{1/p^{\infty}}]/t^{p-1}\equiv_{\exists} \bar \F_p^{\omega}[t^{1/p^{\infty}}]/t^{p-1}$. Let $\phi \in L_r$ be an existential sentence. Then
$$ R_p[t^{1/p^{\infty}}]/t^{p-1}\models \phi \iff (\exists m\in \N)(R_p[t^{1/p^{m}}]/t^{p-1} \models \phi )$$
$$\iff (\exists m\in \N)(R_p[t]/t^{p^m(p-1)} \models \phi )$$ 
Note that for any ring $R$ and $m\in \N$, the ring $R[t]/t^m$ is quantifier-free interpretable in $R$ and also \textit{uniformly} in $R$.\footnote{One has to identify $a_0+...+a_{m-1}t^{m-1}$ with $(a_0,...,a_{m-1})\in R^m$ and check that the ring operations of $R[t]/t^m$ are definable in $L_r$ via this translation and also  \textit{uniformly} in $R$. } This fact together with Proposition \ref{rpexistential} imply that

$$(\exists m\in \N)(R_p[t]/t^{p^m(p-1)} \models \phi )  \iff  (\exists m\in \N)(\bar \F_p^{\omega} [t]/t^{p^m(p-1)} \models \phi )$$
$$\iff \bar \F_p^{\omega}[t^{1/p^{\infty}}]/t^{p-1}\models \phi  $$
The conclusion follows from Lemma \ref{equivomega}.
\end{proof}
As a consequence, we have:
\bc \label{corconv}
Let $n\in \N$ and $f_i(x),g_i(x)\in \Z[x]$ be multi-variable polynomials in $x=(x_1,...,x_m)$, for $i=1,...,n$. We have that
$$\Z^{ab}/p\Z^{ab}\models  \exists x( \bigwedge_{1\leq i\leq n} f_i(x)=0\land g_i(x)\neq 0)$$
$$\iff  \bar \F_p[t^{1/p^{\infty}}]/t^{p-1}\models \exists x_1,...,x_n(\bigwedge_{1\leq i,j\leq n} f_i(x_j)=0\land g_i(x_i)\neq 0)$$
where each $x_j=(x_{j1},...,x_{jm})$ an $m$-tuple.
\ec
\begin{proof}
Immediate from Proposition \ref{resringab}.
\end{proof}

\section{From $\Z^{ab}$ to ACVF} \label{localapprox}
In the previous section we saw how $Th_{\exists}(\Z^{ab}/p\Z^{ab})$ reduces to $Th_{\exists}(  \bar \F_p[t^{1/p^{\infty}}]/t^{p-1})$ (see Corollary \ref{corconv}). We shall now encode the existential theory of $ \bar \F_p[t^{1/p^{\infty}}]/t^{p-1}$ in $L_r$ inside the existential theory of $\bar \F_p((t))^{1/p^{\infty}}$ in $L_{val}$. The latter is well understood by a result of Ansombe-Fehm (see Theorem \ref{AnsFehmequiv}) and is precisely $ACVF_{(p,p),\exists}$. 
\subsection{Positive characteristic valuation rings}
Fix a prime $p\in \mathbb{P}$. We convert existential statements about $\bar \F_p[t^{1/p^{\infty}}]/t^{p-1}$ in $L_r$ to existential statements about $\bar \F_p((t))^{1/p^{\infty}}$ in $L_{val}$ and also \textbf{uniformly} in $p$. The main idea originates in Section 4.1 \cite{KK} and the following Proposition is an easy adaptation of Proposition 4.1.1 \cite{KK}:
\bp \label{localglobal}
Fix $p\in \mathbb{P}$ and let $f_i(x),g_i(x)\in \bar \F_p[x]$ be multi-variable polynomials in $x=(x_1,...,x_m)$ for $i=1,...,n$. Then
$$ \bar \F_p[t^{1/p^{\infty}}]/t^{p-1}\models \exists x \bigwedge_{1\leq i\leq n} (f_i(x)=0\land g_i(x)\neq 0) \iff $$ 
$$\bar \F_p((t))^{1/p^{\infty}} \models  \exists x\in \mathcal{O}_v \bigwedge_{1\leq i,j\leq n}(v(f_i(x))> v(g_j(x)))$$
\ep  

\begin{proof} 
First observe that
$$\bar \F_p[[t]]^{1/p^{\infty}}/t^{p-1}\cong \varinjlim \bar \F_p[[t^{1/p^n}]]/t^{p-1}\cong \varinjlim \bar \F_p[t^{1/p^n}]/t^{p-1}\cong \bar \F_p[t^{1/p^{\infty}}]/t^{p-1} \ \ (\dagger)$$
"$\Rightarrow$": Let $a\in (\bar \F_p[t^{1/p^{\infty}}]/t^{p-1})^m$ be such that $f_i(a)=0\land g_i(a)\neq 0 $, for $i=1,...,n$ and let $\tilde{a}$ be any lift of $a$ in $\bar \F_p[[t]]^{1/p^{\infty}}$ via the isomorphism $(\dagger)$. We see that $v(f_i(\tilde{a}))\geq p-1 >v(g_j(\tilde{a}))$, for all $i,j=1,...,n$.\\
"$\Leftarrow$": Let $b\in (\bar \F_p[[t]]^{1/p^{\infty}})^m$ be such that $v(g_j(b))< v(f_i(b))$ for all $1\leq i,j\leq n$. Set $\gamma_1=max \{v(g_j(b)):j=1,...,n\}$ and $\gamma_2=min \{v(f_i(b)):i=1,...,n\}$ and consider the open interval $I=(\gamma_1,\gamma_2)\subseteq \frac{1}{p^{\infty}}\Z^{\geq 0}$. Since $\frac{1}{p^{\infty}}\Z$ is dense in $\mathbb{R}$, we can find $q\in \frac{1}{p^{\infty}}\Z$ such that $p-1\in qI$. We now make use of the fact that for each $q\in \frac{1}{p^{\infty}}\Z^{>0}$, there is an $\bar \F_p$-embedding 
$$\rho:  \bar \F_p((t))^{1/p^{\infty}}\to \bar \F_p((t))^{1/p^{\infty}}$$
which maps $t\mapsto t^q$. Indeed, if $q\in \frac{1}{p^{N}}\Z^{>0}$ for some $N\in \N$, then there exists an embedding $\rho:\bar \F_p((t))^{1/p^{N}}\to \bar \F_p((t))^{1/p^{N}}$ mapping $t\to t^q$, exactly as in Remark 7.9 \cite{AnscombeFehm}. Such a map can also be extended uniquely to the perfect hull $ \bar \F_p((t))^{1/p^{\infty}}$.\\
Now let $\rho:  \bar \F_p((t))^{1/p^{\infty}}\to \bar \F_p((t))^{1/p^{\infty}}$ be as above. Then, since $f_i(x),g_j(x)\in \bar \F_p[x]$ and $\rho$ fixes $\bar \F_p$, we get
$$v(g_j(\rho(b)))=v(\rho(g_j(b)))=qv(g_j(b))<qv(f_i(b))=v(\rho(f_i(b)))=v(f_i(\rho(b)))$$
for all $1\leq i,j\leq n$. We may thus replace our witness $b$ with $a=\rho(b)$.\\ 
Since $p-1\in qI$, we get $f_i(a)=0\mod t^{p-1}\land g_i(a)\neq 0\mod t^{p-1} $, for all $i=1,...,n$. The reduction of $a$ modulo $t^{p-1}$, seen as a tuple in $ \bar \F_p[t^{1/p^{\infty}}]/t^{p-1}$ via $\dagger$, is the desired witness.
\end{proof}

\begin{rem} \label{ACFrem}
 If we have a positive existential sentence, then one sees that
$$\bar \F_p[t^{1/p^{\infty}}]/t^{p-1}\models \exists x  \bigwedge_{1\leq i\leq n} f_i(x)=0 \iff ACF_p \models \exists x  \bigwedge_{1\leq i\leq n} f_i(x)=0 $$
and no reference to valuations is needed. A similar observation holds true for an existential sentence defined purely by inequations. On the other hand, we see that
$$\bar \F_p[t^{1/p^{\infty}}]/t^{p-1}\models \exists x \bigwedge_{1\leq i\leq n} (f_i(x)=0\land g_i(x)\neq 0)  $$
$$\not \Rightarrow ACF_p\models \exists x \bigwedge_{1\leq i\leq n} (f_i(x)=0\land g_i(x)\neq 0)$$
as the example $\exists x(x^p=1 \land x\neq 1)$ shows. 
\end{rem}

\subsection{Diophantine problems modulo a fixed prime}
We will need a special case of the following Theorem due to Ansombe-Fehm: 
\bt [Corollary 7.2 \cite{AnscombeFehm}] \label{AnsFehmequiv}
Let $(K, v)$ and $(L, w)$ be equicharacteristic henselian non-trivially
valued fields. Then $Th_{\exists}(K, v) = Th_{\exists}(L, w)$ in $L_{val}$ if and only if $Th_{\exists}(Kv) = Th_{\exists}(Lw)$ in $L_r$.
\et 
In that special case, there is a direct argument which was communicated to me by E. Hrushovski. It bears many similarities with the proof of Corollary 7.7 \cite{AnscombeFehm} but ultimately relies on A. Robinson's results on ACVF and not Kuhlmann's theory of tame fields. We refer to Sections 3.4 and 3.5 \cite{vdd} for background material on $ACVF$.
\bob \label{Udi}
We have that $Th_{\exists}(\bar \F_p((t))^{1/p^{\infty}},v_t)=ACVF_{(p,p),\exists}$ in the language of valued fields $L_{val}$. 
\eob 
\begin{proof}
Every finite extension of $(\bar \F_p((t)),v_t)$ will be a complete discrete valued field with residue field $\bar \F_p$ and thus (abstractly) isomorphic to $\bar \F_p((t))$ (see Theorem 2, pg.33 \cite{Ser}). Therefore, the valued field $(\overline{\F_p((t))},v_t)$ is a directed union of structures isomorphic to $(\bar \F_p((t)),v_t)$. It follows that 
$$Th_{\exists}(\bar \F_p((t)),v_t)=Th_{\exists}(\overline{\F_p((t))})=ACVF_{(p,p),\exists}$$   
Moreover, if $(\bar \F_p((t)),v_t)\subset (F,v)\subset (\overline{\F_p((t))},v_t)$ is any intermediate field, then we will also have that $Th_{\exists}(F,v)= ACVF_{(p,p),\exists}$, since by default $Th_{\exists}(\bar \F_p((t)),v_t)\subset Th_{\exists}(F,v)\subset Th_{\exists} (\overline{\F_p((t))},v_t)$ and as we argued above $Th_{\exists}(\bar \F_p((t)),v_t)=Th_{\exists}(\overline{\F_p((t))},v_t)$. In particular, we get that $Th_{\exists}( \bar \F_p((t))^{1/p^{\infty}},v_t)=ACVF_{(p,p),\exists}$, which is what we wanted to show.\\
\end{proof}

\bp \label{localglobalACVF}
Fix $p\in \mathbb{P}$ and let $f_i(x),g_i(x)\in \bar \F_p[x]$ be multi-variable polynomials in $x=(x_0,...,x_m)$ for $i=1,...,n$. Then
$$ \bar \F_p[t^{1/p^{\infty}}]/t^{p-1}\models \exists x \bigwedge_{1\leq i\leq n} (f_i(x)=0\land g_i(x)\neq 0) \iff $$ 
$$ACVF_{(p,p)} \models  \exists x\in \mathcal{O}_v \bigwedge_{1\leq i,j \leq n}(v(f_i(x))> v(g_j(x)))$$
\ep  
\begin{proof}
Immediate from Proposition \ref{localglobal} and Observation \ref{Udi}.
\end{proof}
Recall that the theory $ACVF_{(q,p)}$ is complete in $L_{val}$ by A. Robinson for each characteristic pair $(q,p)$ (see Section 3.5, Corollary 3.34 \cite{vdd}) and thus also decidable, since it is recursively axiomatized. 
\bc \label{resringabdec}
Fix a prime number $p\in \mathbb{P}$. The existential theory of $\Z^{ab}/p\Z^{ab}$ is decidable in $L_r$.
\ec 
\begin{proof}
By Corollary \ref{corconv}, it suffices to show that $ \bar \F_p[t^{1/p^{\infty}}]/t^{p-1} $ is existentially decidable in $L_r$. Once again, using Lemma \ref{DNF}, we may focus on the special case where $\phi$ is a conjunction of literals. The conclusion now follows either by using Proposition \ref{localglobal} and Theorem \ref{AnsFehmequiv} or by Proposition \ref{localglobalACVF} and the fact that $ACVF_{(p,p)}$ is (existentially) decidable in $L_{val}$ due to A. Robinson.
\end{proof}
In Question \ref{zabmodn} we formulate two unsolved variants.

\subsection{Diophantine problems modulo all primes}
It is not difficult now to put all the pieces together:
\begin{Theore} \label{mainagain}
The existential theory of $\{\Z^{ab}/p\Z^{ab}:p\in \mathbb{P}\}$ is decidable in $L_r$.
\end{Theore}
\begin{proof}
By Lemma \ref{DNF}, we may focus our attention on sentences $\phi \in L_r$ of the form
$$\phi= \exists x( \bigwedge_{1\leq i\leq n} f_i(x)=0\land g_i(x)\neq 0)$$
where $x=(x_1,...,x_m)$, for some $m\in \N$. For every $p\in \mathbb{P}$, we have by Corollary \ref{corconv} and Proposition \ref{localglobalACVF} that $\Z^{ab}/p\Z^{ab}\models \phi\iff ACVF_{(p,p)}\models \phi^*$, where
$$\phi^*=\exists x_1,...,x_n \in \mathcal{O}_v [\bigwedge_{1\leq i,j,k \leq n}(v(f_i(x_j))> v(g_k(x_k)))]$$
Note that $\phi^*$ depends only on $\phi$ and \textit{not} on $p$. Now a standard argument may be used to check if $ACVF_{(p,p)}\models \phi^*$ for all $p\in \mathbb{P}$:\\ 
First check if $ACVF_{(0,0)} \models \phi^*$, which can be done recursively since $ACVF_{(0,0)}$ is decidable in $L_{val}$. If we get a positive answer, an application of G\"odel's completeness theorem gives $ACVF_{(p,p)}\models \phi^*$ for $p\geq p_0$, where $p_0\in \mathbb{P}$ is \textit{computable}. It suffices to check that $ACVF_{(p,p)}\models \phi^*$ for finitely many $p\leq p_0$, which is again possible since for each individual $p\in \mathbb{P}$ the theory $ACVF_{(p,p)}$ is decidable. If we get a negative answer, then by completeness of $ACVF_{(0,0)}$, we will have that $ACVF_{(0,0)} \models \lnot \phi^*$, whence $ACVF_{(p,p)}\models \lnot \phi^*$, for all but finitely many $p\in \mathbb{P}$.
\end{proof}
\begin{rem}
Using Remark \ref{ACFrem}, we see that $T=Th_{\exists^+}(\{\Z^{ab}/p\Z^{ab}:p\in \mathbb{P}\})$ simply reduces to $ACF$, i.e. 
$$T\models  \exists x( \bigwedge_{1\leq i\leq n} f_i(x)=0) \iff ACF\models \exists x( \bigwedge_{1\leq i\leq n} f_i(x)=0) $$
and no reference to valuations is needed. 
\end{rem}

\section{Open problems}
The following two questions are of central importance:
\begin{question}
$(a)$ Is $Th(\Z^{ab})$ undecidable in $L_r$?\\
$(b)$ Is $Th_{\exists}(\Z^{ab})$ undecidable in $L_r$?
\end{question}
For some comments on the \textit{conjectural} undecidability of $\Z^{ab}$, see the last bullet point on Section 6.3, pg.191 in \cite{JK}. If one wants to explore aspects of $\Z^{ab}$ which are more likely to be decidable, one may ask the following natural extensions of the results presented in this paper:

\bq \label{zabmodn}
$(a)$ Is  $Th(\Z^{ab}/p\Z^{ab})$ decidable in $L_r$? How about $Th(\{\Z^{ab}/p\Z^{ab}:p\in \mathbb{P}\})$ in $L_r$? \\
$(b)$ Let $n\in \N$. Is $Th_{\exists}(\Z^{ab}/n\Z^{ab})$ decidable in $L_r$?
\eq 
By the Chinese Remainder Theorem \ref{CRT},  we see that Question \ref{zabmodn}$(b)$ is reduced to the case of $\Z^{ab}/p^n\Z^{ab}$ for $p\in \mathbb{P}$ and $n\in \N$.

\section*{Acknowledgements}
I am grateful to the anonymous referee for various local corrections and suggestions for improving the exposition.
\bibliographystyle{alpha}
\bibliography{references2}
\Addresses
\end{document}